\documentclass[12pt]{amsart}
\usepackage{latexsym}
\usepackage{amssymb} 
\usepackage{mathrsfs}
\usepackage{amsmath}
\usepackage{latexsym}
\usepackage{delarray}
\usepackage{amssymb,amsmath,amsfonts,amsthm,mathrsfs}

\usepackage{hyperref}
\renewcommand\eqref[1]{(\ref{#1})} 

 \newtheorem{thm}{Theorem}[section]
 \newtheorem{cor}[thm]{Corollary}
 \newtheorem{lem}[thm]{Lemma}
 \newtheorem{prop}[thm]{Proposition}
 \theoremstyle{definition}
 \newtheorem{defn}[thm]{Definition}
 \theoremstyle{remark}
 \newtheorem{rem}[thm]{Remark}
 
 \numberwithin{equation}{section}
\newcommand{\half}{\frac{1}{2}}

\newcommand{\ene}{\mathbb{N}}

\def\SU2{{{\rm SU(2)}}}
\newcommand{\er}{\mathbb{R}}
\newcommand{\ce}{\mathbb{C}}

\newcommand{\tn}{\mathbb{T}^n}

\newcommand{\ern}{{\mathbb{R}}^n}

\newcommand{\bi}{\begin{itemize}}
\newcommand{\ei}{\end{itemize}}
\newcommand{\be}{\begin{enumerate}}
\newcommand{\ee}{\end{enumerate}}
\newcommand{\beq}{\begin{equation}}
\newcommand{\eq}{\end{equation}}

\newcommand{\lap}{\mathcal{L}_G}
\newcommand{\lapm}{\Delta_M}

\newcommand{\cdxi}{\ce^{d_{\xi}\times d_{\xi}}}
\def\SO3{{{\rm SO(3)}}}

\newcommand{\cinfm}{\mathcal{C}^{\infty}(M)}

\def\p#1{{\left({#1}\right)}}

\def\jp#1{{\left\langle{#1}\right\rangle}}

\def\Op{{{\rm Op}}}

\DeclareMathOperator{\Tr}{Tr}

\def\Gh{{\widehat{G}}}

\def\HS{{\mathtt{HS}}}
\def\Rn{{{\mathbb R}^n}}
\def\Tn{{{\mathbb T}^n}}

\def\T{{{\mathbb T}^1}}
\def\TT{{{\mathbb T}^2}}
\def\N{{{\mathbb N}}}
\def\C{{{\mathbb C}}}
\def\SU2{{{\rm SU(2)}}}
\def\lapsu2{{{\mathcal L}_\SU2}}

\def\va{\varphi}

\def\Op{\text{\rm Op}}

\newcommand{\subll}{\mathcal{L}_{sub}}

\begin{document}

%
\title[Schatten classes on compact manifolds]
 {Schatten classes on compact manifolds: Kernel conditions}

\author[Julio Delgado]{Julio Delgado}

\address{%
Department of Mathematics\\
Imperial College London\\
180 Queen's Gate, London SW7 2AZ\\
United Kingdom
}

\email{j.delgado@imperial.ac.uk}

\thanks{The first author was supported by a Marie Curie International Incoming
Fellowship within the 7th European Community Framework Programme
under grant  PIIF-GA-2011-301599. The second author was supported by EPSRC Leadership Fellowship EP/G007233/1 and EPSRC grant EP/K039407/1. }
\author[Michael Ruzhansky]{Michael Ruzhansky}

\address{%
Department of Mathematics\\
Imperial College London\\
180 Queen's Gate, London SW7 2AZ\\
United Kingdom
}

\email{m.ruzhansky@imperial.ac.uk}



\date{\today}
\begin{abstract}
In this paper we give criteria on integral  kernels ensuring that integral operators on compact manifolds
belong to Schatten classes. 
 A specific test for nuclearity is established as well as the corresponding trace formulae. 
 In the special case of compact Lie groups, kernel
 criteria in terms of (locally and globally) hypoelliptic operators are also given.  
\end{abstract}

\maketitle
\section{Introduction}
Given a closed smooth manifold $M$ (smooth manifold without boundary)
endowed with a positive measure $dx$, in this paper we give 
sufficient conditions on Schwartz integral kernels in order to ensure that 
the corresponding integral operators belong to
different Schatten classes. The problem of finding such criteria on different kinds of domains is classical and has been much studied, e.g. the paper \cite{bs:sing} by Birman and Solomyak is a good introduction to the subject. In particular, it is well known that the smoothness of the kernel is related to the behaviour of the singular numbers. 
 
In this paper we present criteria for Schatten classes and, in particular, for the trace class 
operators on compact smooth manifolds without boundary. Compact Lie groups will also be considered as a special case since then additional results can be obtained, also
allowing criteria in terms of hypoelliptic operators such as the sub-Laplacian. 
The sufficient conditions on 
integral kernels $K(x,y)$ for Schatten classes will require regularity of a certain order in
either $x$ or $y$, or both. 

We note that already some results of Birman and Solomyak
 \cite{bs:sing} can be extended to compact manifolds but our approach allows one to
be flexible about sets of variables in which one imposes the regularity of the kernel.

In order to obtain criteria for general Schatten classes we will use the well-known method of factorisation, particularly in the way applied by O'Brien in \cite{OBrien:test}.  For applications to trace formulas of Schr{\"o}dinger operators see also \cite{OBrien:test2}.
 
Schatten classes of pseudo-differential operators in the setting of the Weyl-H\"or\-man\-der calculus have been considered in 
\cite{Toft:Schatten-AGAG-2006}, \cite{Toft:Schatten-modulation-2008}, \cite{Buzano-Toft:Schatten-Weyl-JFA-2010}. Schatten classes on compact Lie groups and $s$-nuclear operators on $L^p$ spaces from the point of view of symbols have been respectively studied by the authors in \cite{dr13:schatten} and \cite{dr13a:nuclp}. In the subsequent part of the 
present paper we establish the
characterisation of Schatten classes on closed manifolds in terms of symbols that
we will introduce for this purpose.

In his classical book (cf. \cite{sugiura:book},  Prop 3.5, page 174) Mitsuo Sugiura gives a trace class criterion for integral operators on $L^2(\T)$ with $C^2$-kernels. More precisely, the theorem asserts that every kernel in $C^2(\TT)$ begets a trace class operator on $L^2(\T)$: if $K(\theta,\phi)$ is a $C^2$-function on $\TT$, then the integral operator $L$ on 
 $L^2(\T)$ defined by
\begin{equation}\label{EQ:Sugiura-trace-op} 
 (Lf)(\theta)=\int\limits_0^{2\pi}K(\theta,\phi)f(\phi)d\phi,
\end{equation}
is trace class and has the trace
\begin{equation}\label{EQ:Sugiura-trace}
\Tr(L)=\frac{1}{2\pi}\int\limits_0^{2\pi}K(\theta,\theta)d\theta.
\end{equation}
The proof of this result relies on the connection between the absolute convergence of Fourier coefficients of the kernel and the trace class property (traceability) of the corresponding operator.

However, in this paper we show that such type of results can be significantly improved by 
using a different approach.
Associating a discrete Fourier Analysis to an elliptic operator on a  compact manifold, we will establish the aforementioned connection in the setting of general closed manifolds, also weakening the known
assumptions on the kernel for the operator to be trace class and for the trace formula
\eqref{EQ:Sugiura-trace} to hold. 
Thus, in this respect, a direct extension of the method employed by Sugiura
leads to weaker results than our approach, for closed manifolds of dimension higher than $2$, and we discuss this 
at the end of Section \ref{SEC:trace-class}.
 
      
To formulate the notions more precisely, let $H$ be a complex Hilbert space endowed with an inner product denoted 
by $\p{\cdot,\cdot}$, and let $T:H\rightarrow H$ be a linear compact operator. If we denote by $T^*:H\rightarrow H$ the adjoint of  $T$, then the linear operator $(T^*T)^\half:H\rightarrow H$ is positive and compact. Let $(\psi_k)_k$ be an orthonormal basis for $H$ consisting of eigenvectors of $|T|=(T^*T)^\half$, and let $s_k(T)$ be the eigenvalue corresponding to the eigenvector 
$\psi_k$, $k=1,2,\dots$. The non-negative numbers $s_k(T)$, $k=1,2,\dots$, are called the singular values of $T:H\rightarrow H$. 
If 
$$
\sum_{k=1}^{\infty} s_k(T)<\infty ,
$$
then the linear operator $T:H\rightarrow H$ is said to be in the {\em trace class} $S_1$. It can be shown that  $S_1(H)$ is a Banach space in which the norm $\|\cdot\|_{S_1}$ is given by 
$$
\|T\|_{S_1}= \sum_{k=1}^{\infty} s_k(T),\,T\in S_1,
$$
multiplicities counted.
Let $T:H\rightarrow H$ be an operator in $S_1(H)$ and let  $(\phi_k)_k$ be any orthonormal basis for $H$. Then, the series $\sum\limits_{k=1}^{\infty} \p{T\phi_k,\phi_k}$   is absolutely convergent and the sum is independent of the choice of the orthonormal basis $(\phi_k)_k$. Thus, we can define the trace $\Tr(T)$ of any linear operator
$T:H\rightarrow H$ in $S_1$ by 
$$
\Tr(T):=\sum_{k=1}^{\infty}\p{T\phi_k,\phi_k},
$$
where $\{\phi_k: k=1,2,\dots\}$ is any orthonormal basis for $H$. If the singular values
are square-summable $T$ is called a {\em Hilbert-Schmidt} operator. It is clear that every trace class operator is a Hilbert-Schmidt operator. More generally, if $0<p<\infty$ and the sequence of singular values is $p$-summable, then $T$ 
is said to belong to the Schatten class  ${S}_p(H)$, and it is well known that each ${S}_p(H)$ is an ideal in $\mathcal{L}(H)$. If $1\leq p <\infty$, a norm is associated to ${S}_p(H)$ by
 \[
 \|T\|_{S_p}=\left(\sum\limits_{k=1}^{\infty}(s_k(T))^p\right)^{\frac{1}{p}}.
 \] 
If $1\leq p<\infty$ 
 the class $S_p(H)$ becomes a Banach space endowed by the norm $\|T\|_{S_p}$. If $p=\infty$ we define $S_{\infty}(H)$ as the class of bounded linear operators on $H$, with 
$\|T\|_{S_\infty}:=\|T\|_{op}$, the operator norm.
In the case $0<p<1$  the quantity $\|T\|_{S_p}$ only defines
 a  quasinorm, and $S_p(H)$ is also complete.

The Schatten classes are nested, with
\begin{equation}\label{EQ:Sch-nested}
{S}_p\subset {S}_q,\,\,\textrm{ if }\,\, 0<p<q\leq\infty,
\end{equation}
and satisfy the important multiplication property (cf. \cite{horn}, \cite{sim:trace}, \cite{gokr})
\beq\label{inn}{S}_p{S}_q\subset S_r,\eq
where
\[\frac{1}{r}=\frac{1}{p}+\frac{1}{q},\qquad 0<p<q\leq\infty.\]
We will apply (\ref{inn}) for factorising our operators $T$ in the form $T=AB$ with $A\in S_p$ and $B\in S_q$, 
and from this we deduce that $T\in S_r$. 

A nice basic introduction to the study of the trace class is included in the book \cite{lax:fa} by Peter Lax. For the  basic theory of Schatten classes we refer the reader to \cite{gokr}, \cite{r-s:mp}, \cite{sim:trace}, \cite{sch:id}. 
 
In this paper we consider integral operators which is not restrictive in view of
the Schwartz integral kernel theorem on closed manifolds. 
If $H=L^2({\Omega},{\mathcal{M}},\mu)$, it is well known that $T$ is a Hilbert-Schmidt operator if and only if  $T$ can be represented by an integral kernel $K=K(x,y)\in L^2(\Omega\times\Omega,\mu\otimes\mu)$. In this paper we are interested in the case when $\Omega$ is a closed manifold (which we denote by $M$).
In particular, we note that in view of \eqref{EQ:Sch-nested} the condition 
$K\in L^2(M\times M)$ implies that $T\in S_p$ for all $p\geq 2$.

For $p<2$, the situation is much more subtle, and
the Schatten classes $S_p(L^2)$ cannot be characterised as in the case $p=2$ by a property analogous to the square integrability of integral kernels. This is a crucial fact that we now briefly describe. A classical result of 
Carleman \cite{car:ex} from 1916 gives the construction of a periodic {\em continuous} function 
$\varkappa(x)=\sum\limits_{k=-\infty}^{\infty}c_k e^{2\pi ikx}$ for which the Fourier coefficients $c_k$ satisfy
\begin{equation}\label{EQ:Carleman}
\sum\limits_{k=-\infty}^{\infty} |c_k|^p=\infty\qquad \textrm{ for any } p<2.
\end{equation}
Now, using this and considering the normal operator 
\begin{equation}\label{EQ:Carleman2}
Tf=f*\varkappa
\end{equation}
acting on $L^2(\T)$ 
one obtains that the sequence $(c_k)_k$ forms a complete system of eigenvalues of 
this operator corresponding to the complete orthonormal system $\phi_k(x)=e^{2\pi ikx}$, 
$T\phi_k=c_k\phi_k$. The system $\phi_k$ is also complete for $T^*$, $T^*\phi_k=\overline{c_k}\phi_k$, 
so that the singular values of $T$ are given by $ s_k(T)=|c_k|$, and
hence by \eqref{EQ:Carleman} we have
$$\sum\limits_{k=-\infty}^{\infty}s_k(T)^p=\infty \qquad \textrm{ for any } p<2.$$   
 In other words, in contrast to the case of the class ${S}_2$ of Hilbert-Schmidt operators which is characterised by the square integrability of the kernel, Carleman's result shows that below the index $p=2$ the class of kernels generating operators in the Schatten class $S_p$  cannot be characterised by criteria of  the type $ \iint |K(x,y)|^\alpha dxdy<\infty ,$
 since the kernel $K(x,y)=\varkappa (x-y)$ of the operator $T$ in
 \eqref{EQ:Carleman2} satisfies any kind of integral condition of such form 
 due to the boundedness of $\varkappa$. 

This example demonstrates the relevance of obtaining criteria for operators to belong to Schatten classes
 for $p<2$ and, in particular, motivates the results in this paper.  Among other things,
 we may also note that the continuity of the kernel (as in the above example)
 also does not guarantee that the operator would belong to any of the Schatten
 classes $S_p$ with $p<2$. Therefore, it is natural to ask what regularity imposed on the
 kernel would guarantee such inclusions (for example, the $C^2$ condition in Sugiura's
 result mentioned earlier does imply the traceability on $\mathbb T^1$). Thus, these
 questions will be the main interest of the present paper.

As for criteria for operators to belong to Schatten classes $S_{p}$ for $0<p<2$, a simplified
version of our results for kernels in Sobolev spaces can be stated as follows:

\begin{thm}\label{THM:simplified}
Let $M$ be a closed manifold of dimension $n$. 
Let $K\in H^{\mu}(M\times M)$ for some $\mu>0$. Then
the integral operator $T$ on 
 $L^2(M)$, defined by
 \[(Tf)(x)=\int\limits_{M} K(x,y)f(y)dy,\]
is in the Schatten classes $S_p(L^2(M))$ for $p>\frac{2n}{n+2\mu}$.
In particular, if $\mu>\frac n2$, then $T$ is trace class.
\end{thm}
This results improves, for example, Sugiura's result for the operator
\eqref{EQ:Sugiura-trace-op}. Theorem \ref{THM:simplified} follows
from the main result Theorem \ref{ext322} giving criteria in terms of the mixed Sobolev
spaces, Proposition \ref{PROP:mixed-Sobolev}, and Corollary \ref{mix3}.
In particular, the use of mixed Sobolev spaces in Theorem \ref{ext322} allows
us to formulate criteria requiring different (smaller) regularities of $K(x,y)$ in $x$ and $y$, 
or only in one of these variables.
 
 We note that the situation for Schatten classes $S_p$ for $p>2$ is simpler and, 
 in fact, similar to that of
 $p=2$. For example, for left-invariant operators on compact Lie groups $G$, i.e. for convolution operators of the form
 $Tf=f*\varkappa$, it was shown in \cite{dr13:schatten} that the condition $\varkappa\in L^{p'}(G)$, $1\leq p'\leq 2$,
 implies that $T\in S_p(L^2(G))$, where $\frac1{p'}+\frac1p=1$. The converse of this is also true but for 
 interchanged indices, i.e. the condition $T\in S_p(L^2(G))$ but now for $1\leq p\leq 2$ implies that
 $\varkappa\in L^{p'}(G)$. We refer to \cite{dr13:schatten} for this as well as for the symbolic characterisation
 of Schatten classes in the setting of compact Lie groups.
 
In this work we allow singularities in the kernel so that the formula \eqref{EQ:Sugiura-trace}
would need to be modified in order for the integral over the diagonal to make sense. In such case, in order to calculate the trace of an integral operator using a non-continuous kernel along the diagonal, one idea is to average it to obtain an integrable function. Such an averaging process has been introduced by Weidmann \cite{weid:av} in the Euclidean setting, and applied by Brislawn in \cite{bri:k1}, \cite{bri:k2} 
for integral operators on $L^2(\ern)$ and on $L^2({\Omega},{\mathcal{M}},\mu)$, respectively, where $\Omega$ is a second countable topological space endowed with a $\sigma$-finite Borel measure. The corresponding extensions to the $L^p$ setting have been established in \cite{del:trace} and \cite{del:tracetop}. The $L^2$ regularity of such an averaging process
 is a consequence of the $L^2$-boundedness of the martingale maximal function. Denoting by $\widetilde{K}(x,x)$ the pointwise values of this averaging process, Brislawn  \cite{bri:k2} proved the following formula for a trace class operator $T$ on $L^2(\mu)$, for the extension to $L^p$ see \cite{del:trace}:
\beq\label{f1} \Tr(T)=\int_{\Omega} \widetilde{K}(x,x)d\mu(x).\eq

In Section \ref{SEC:Fourier} we describe the discrete Fourier analysis involved in our problem and establish several relations between eigenvalues and their multiplicities for elliptic positive pseudo-differential operators.
Then in Section \ref{SEC:Schatten-classes} we establish our criteria for Schatten classes on compact manifolds and, in particular, for the trace class in Section \ref{SEC:trace-class}. 
For this, we briefly recall the definition of the averaging process involved in the formula (\ref{f1}).  
We also explain another method relating the convergence of the Fourier coefficients of the kernel with the traceability. In Section \ref{SEC:Lie-groups} the special case of compact Lie groups is considered where we show that the criteria can be also given using hypoelliptic operators. 

The authors would like to thank V\'{e}ronique Fischer and Jens Wirth for discussions
and remarks.

\section{Fourier analysis associated to an elliptic operator}  
\label{SEC:Fourier}

In this section we start by recording some basic elements of Fourier analysis on compact manifolds which will be useful for our analysis.   

Let $M$ be a compact smooth manifold of dimension $n$ without boundary, endowed with a fixed volume $dx$.  
 We denote by $\Psi^{\nu}(M)$ the H\"ormander class of pseudo-differential 
 operators of order $\nu\in\er$,
 i.e. operators which, in every coordinate chart, are operators in H\"ormander classes 
 on $\Rn$ with symbols
 in $S^\nu_{1,0}$, see e.g. \cite{shubin:r} or \cite{rt:book}.
 In this paper we will be using the class  $\Psi^{\nu}_{cl}(M)$ of classical operators, i.e. operators
 with symbols having (in all local coordinates) an asymptotic expansion of the symbol in
 positively homogeneous components (see e.g. \cite{Duis:BK-FIO-2011}).
 Furthermore, we denote by $\Psi_{+}^{\nu}(M)$ the class of positive definite operators in 
 $\Psi^{\nu}_{cl}(M)$,
 and by $\Psi_{e}^{\nu}(M)$ the class of elliptic operators in $\Psi^{\nu}_{cl}(M)$. Finally, 
 $$\Psi_{+e}^{\nu}(M):=\Psi_{+}^{\nu}(M)\cap \Psi_{e}^{\nu}(M)$$ 
 will denote the  class of classical positive elliptic 
 pseudo-differential operators of order $\nu$.
 We note that complex powers of such operators are well-defined, see e.g.
 Seeley \cite{Seeley:complex-powers-1967}. 
 In fact, all pseudo-differential operators considered in 
 this paper will be classical, so we may omit explicitly mentioning it every time, but we note
 that we could equally work with general operators in $\Psi^{\nu}(M)$ since their
 powers have similar properties, see e.g. \cite{Strichartz:functional-calculus-AJM-1972}.
 
We now associate a discrete Fourier analysis to the operator $E\in\Psi_{+e}^{\nu}(M)$ inspired by those
considered by Seeley (\cite{see:ex}, \cite{see:exp}), see also \cite{Greenfield-Wallach:hypo-TAMS-1973}. 
However, we adapt it to our purposes and prove several auxiliary statements concerning the
eigenvalues of $E$ and their multiplicities, useful to us in the sequel.

The eigenvalues of $E$ form a sequence $\{\lambda_j\}$, with multiplicities taken into account. 
For each eigenvalue $\lambda_j$, there is
the corresponding finite dimensional eigenspace $F_j$ of functions on $M$, which are smooth due to the 
ellipticity of $E$. We set 
$$
d_j:=\dim F_j, 
\textrm{ and } F_0:=\ker E, \; \lambda_0:=0.
$$
We also set $d_{0}:=\dim F_{0}$. Since the operator $E$ is elliptic, it is Fredholm,
hence also $d_{0}<\infty$ (we can refer to \cite{Atiyah:global-aspects-1968} for
various properties of $F_{0}$ and $d_{0}$).

We fix  an orthonormal basis of $L^2(M,dx)$ consisting of eigenfunctions of $E$:
\beq\label{fam}
\{e^k_j\}_{j\geq 0}^{1\leq k\leq d_j},
\eq 
where $\{e^k_j\}^{1\leq k\leq d_j}$ is an orthonormal basis of $F_j$. 
We denote by $(\cdot,\cdot )$ the standard inner product on $L^2(M)$ associated to its volume element. 
Let $P_j:L^2(M)\rightarrow F_j$ be the corresponding orthogonal projections on $F_j$. 
 We observe that 
 \[P_jf=\sum\limits_{k=1}^{d_j}( f,e_j^k) e_j^k,\]
for $f\in L^2(M)$. 
The Fourier inversion formula takes the form 
\begin{equation}\label{EQ:FI-L2}
f=\sum\limits_{j=0}^{\infty}P_jf=\sum\limits_{j=0}^{\infty}\sum\limits_{k=1}^{d_j}(f,e_j^k ) e_j^k,
\end{equation}
for each $f\in L^2(M)$, and where the convergence is understood with respect to the $L^2(M)$-norm. 

\begin{defn}\label{def1} Let $u\in\mathcal{D}'(M)$ and let 
$j\in\ene_{0}=\ene\cup\{0\}$ be a nonnegative integer.
We define $\widehat{u}(j)\in F_j^*$ by $\widehat{u}(j)(\va):=u(\va)$, for $\va\in F_j$. 
\end{defn}

For the distributional valuations we use the notation $u(\va)$ or $\langle u, \va\rangle$. If $\varphi\in \mathcal{D}(M)$ and $u\in\mathcal{D}'(M)$ we have
 \[
 \varphi=\sum\limits_{j=0}^{\infty}\sum\limits_{k=1}^{d_j}( \varphi,e_j^k) e_j^k,
 \]
and 
\beq 
u(\varphi)=\sum\limits_{j=0}^{\infty}\sum\limits_{k=1}^{d_j}(\varphi,e_j^k) u(e_j^k).
\label{coef1}
\eq
We note that the same type of formula holds for operators from $\cinfm$ to $\mathcal{D}'(M)$, namely, if
$T:\cinfm\to\mathcal{D}'(M)$ is a linear continuous operator, we have
\beq 
T(\varphi)=\sum\limits_{j=0}^{\infty}\sum\limits_{k=1}^{d_j}(\varphi,e_j^k) T(e_j^k),
\label{coef2}
\eq
with an appropriate distributional understanding of convergence.
The Fourier coefficients  of $u\in\mathcal{D}'(M)$ associated to the basis (\ref{fam}) can be obtained from (\ref{coef1}):
\[\widehat{u}(j)(\va)=\sum\limits_{k=1}^{d_j}(\va,e_j^k)u(e_j^k) .\]
In particular, if $u\in L^2(M)$ we obtain
\[\widehat{u}(j)(\va)=\sum\limits_{k=1}^{d_j}(\va,e_j^k)(u,\overline{e_j^k}).\]
From the paragraph above we can deduce:  

\begin{prop} If $u\in\mathcal{D}'(M)$ then $\widehat{u}(j)\circ P_j$ is in ${\mathcal{D}}'(M)$, and 
\begin{equation}\label{EQ:FI-dist}
u=\sum\limits_{j=0}^\infty\widehat{u}(j)\circ P_j \textrm{ in }\mathcal{D}'(M).
\end{equation}
\end{prop}

\begin{proof} If $\varphi\in\mathcal{D}(M)$ is a test function we have  
\begin{multline*}
\sum\limits_{j= 0}^\infty(\widehat{u}(j)\circ P_j)(\varphi)=\sum\limits_{j\geq 0}\widehat{u}(j)(P_j\varphi) 
=\sum\limits_{j=0}^\infty\widehat{u}(j)\sum\limits_{k=1}^{d_j}(\varphi,e_j^k) e_j^k\\
=\sum\limits_{j=0}^\infty\sum\limits_{k=1}^{d_j}(\varphi,e_j^k)\widehat{u}(j)(e_j^k)
=\sum\limits_{j=0}^\infty\sum\limits_{k=1}^{d_j}(\varphi,e_j^k)u(e_j^k)
=u(\varphi),
\end{multline*}
in view of \eqref{coef1},
completing the proof.
\end{proof}

Comparing the Fourier inversion formula \eqref{EQ:FI-L2} in $L^2(M)$ with
the formula \eqref{EQ:FI-dist} in $\mathcal{D}'(M)$, for a function $f\in L^2(M)$,
we can identify its distributional Fourier coefficients $\widehat{f}(j)$ with their
action on the basis of $F_j$ given by 
$$
\widehat{f}(j,k):=(f,e_j^k).
$$
We will also denote 
sometimes by $\mathcal{F}$ the Fourier transform associating to $f\in L^2(M)$ its Fourier coefficients.
Since $\{e^k_j\}_{j\geq 0}^{1\leq k\leq d_j}$ forms a complete orthonormal system in $L^2(M)$, for all $f\in L^2(M)$ we have the Plancherel formula
\beq \label{EQ:Plancherel}
\|f\|^2_{L^2(M)}=\sum\limits_{j=0}^{\infty}\sum\limits_{k=1}^{d_j}|(f,e_j^k) _{L^2}|^2=
\sum\limits_{j=0}^{\infty}\sum\limits_{k=1}^{d_j}|\widehat{f}(j,k)|^2.
\eq
Since the criteria that we will obtain depend (a-priori) 
on the choice of the orthonormal basis $\{e^k_j\}$, 
the asymptotics of the corresponding eigenvalues play an essential role.  
We now establish several simple but useful relations between the eigenvalues $\lambda_j$
and their multiplicities $d_j$. 

\begin{prop} 
Let $M$ be a closed manifold of dimension $n$, and let $E\in\Psi_{+e}^{\nu}(M)$, with $\nu>0$. 
Then there exists a constant $C>0$ such that
\begin{equation}\label{EQ:Weyl}
d_j\leq C (1+\lambda_j)^{\frac{n}{\nu}}
\end{equation}
for all $j\geq 1$.
Moreover, we have
\beq
\label{asymp2}
\sum\limits_{j=1}^{\infty}d_j(1+\lambda_j)^{-q}<\infty \quad{\textrm{ if and only if }}\quad q>\frac{n}{\nu}.
\eq
\end{prop}

\begin{proof}
We observe that $(1+\lambda_j)^{1/\nu}$ is an eigenvalue of the first-order elliptic positive
operator $(I+E)^{1/\nu}$ of multiplicity $d_j$. The Weyl formula for the eigenvalue counting
function for the operator $(I+E)^{1/\nu}$ yields
$$
\sum_{j:\ (1+\lambda_j)^{1/\nu}\leq \lambda} d_j=C_0\lambda^n+O(\lambda^{n-1})
$$
as $\lambda\to\infty$. This implies $d_j\leq C(1+\lambda_j)^{n/\nu}$ for sufficiently large
$\lambda_j$, implying \eqref{EQ:Weyl}.

We now prove \eqref{asymp2}.
Let us denote $T:=(I+E)^{-q/2}$. The eigenvalues of $T$ are $(1+\lambda_j)^{-q/2}$ with
multiplicities $d_j$, therefore, we obtain
\begin{equation}\label{EQ:HSK}
\sum\limits_{j=0}^{\infty}d_j(1+\lambda_j)^{-q}=\|T\|^2_{S_{2}}\asymp \|K\|^2_{L^2(M\times M)}.
\end{equation}
At the same time, by the functional calculus of pseudo-differential operators, we know that
$T\in \Psi^{-\nu q/2}(M)$, so that its kernel $K(x,y)$ is smooth for $x\not=y$, and near the diagonal
$x=y$, identifying points with their local coordinates, it satisfies
the estimate 
$$
|K(x,y)|\leq C_\alpha |x-y|^{-\alpha},
$$
for any $\alpha>n-\nu q/2$, 
see e.g. \cite{Duis:BK-FIO-2011} or \cite[Theorem 2.3.1]{rt:book}, and the order is sharp with respect to 
the order of the operator.
Thus, we get that $K\in L^2(M\times M)$ if and only if we can choose $\alpha$ such that
$n>2\alpha>2n-\nu q$, which together with \eqref{EQ:HSK} implies \eqref{asymp2}.
\end{proof}

\section{Schatten classes on compact manifolds}  
\label{SEC:Schatten-classes}

Before stating our first result, we point out that a look at the proof of 
\eqref{EQ:Sugiura-trace} (cf. \cite{sugiura:book}, Prop 3.5) shows 
that that statement can be already improved in the following way:
\begin{prop} \label{ext1} 
Let $\Delta=\frac{\partial ^2}{\partial\theta ^2}+\frac{\partial ^2}{\partial\phi ^2} $ be the Laplacian on $\TT$. 
Let $K(\theta,\phi)$ be a measurable function on $\TT$ and suppose that there exists an integer $q>1$ such that 
$\displaystyle{\Delta^{\frac{q}{2}}}K\in L^2(\T\times\T)$. Then the integral operator $L$ on 
 $L^2(\T)$, defined by
 \[(Lf)(\theta)=\int\limits_0^{2\pi}K(\theta,\phi)f(\phi)d\phi,\]
is trace class and has the trace
\[\Tr(L)=\frac{1}{2\pi}\int\limits_0^{2\pi}\widetilde{K}(\theta,\theta)d\theta,\]
where $\widetilde{K}$ stands for the averaging process described in Section \ref{SEC:trace-class}.
\end{prop}
Our criteria for Schatten classes will also depend on a test of square integrability operating on the kernels 
through an elliptic operator, and the result of Proposition \ref{ext1} will be
improved in Theorem \ref{ext322} 
(see specifically Corollary \ref{mix3})
by using a different approach to the problem. 
In the auxiliary next lemma 
 we show that such condition is independent of the choice of an elliptic operator.
 
\begin{lem} \label{invellipt} 
Let $M$ be a closed manifold. Let $E_1, E_2\in\Psi_{e}^{\nu}(M)$ with $\nu\in\mathbb R$
and let $h\in {\mathcal D}'(M)$.
Then ${E_1}h\in L^2(M)$ if and only if $E_2h\in L^2(M)$.
\end{lem}

\begin{proof} 
Let us suppose that ${E_1}h\in L^2(M)$ and consider a parametrix $L_1\in\Psi_{}^{-\nu}(M)$ of  ${E_1}$. In particular, there exists $R\in \Psi^{-\infty}(M)$ such that 
\[L_1E_1-R=I.\]
Then
\begin{align*} E_2h=&E_2(L_1E_1-R)h\\
=&(E_2L_1)(E_1h)-(E_2R)h,
\end{align*}
with $E_2L_1\in\Psi^{0}(M),  E_2R\in\Psi^{-\infty}(M)$. Since $E_1h\in L^2(M)$ we obtain $(E_2L_1)(E_1h)\in L^2(M)$; the fact that $E_2R$ is smoothing gives us $(E_2R)h\in L^2(M)$, implying that
$E_2 h\in L^2(M)$.
\end{proof}   

We first establish a simple observation for powers of positive elliptic operators to belong to
Schatten classes $S_{p}$ on $L^{2}(M)$.

\begin{prop}\label{PROP:elliptic}
Let $M$ be a closed manifold of dimension $n$, and let $E\in \Psi^{\nu}_{+e}(M)$ be a
positive elliptic pseudo-differential operator of order $\nu>0$.
Let $0<p<\infty$.
Then 
\begin{equation}\label{EQ:ell-powers}
(I+E)^{-\alpha}\in S_{p}(L^{2}(M)) \;\textrm{ if and only if }\;
\alpha>\frac{n}{p \nu}.
\end{equation}
\end{prop}
\begin{proof}
Let $\lambda_{j}$ denote the eigenvalues of $E$, each $\lambda_{j}$ having the multiplicity
$d_{j}$.
Then the operator $(I+E)^{-\alpha}$ is positive definite, 
its singular values are $(1+\lambda_{j})^{-\alpha}$ with multiplicities 
$d_{j}$. Therefore, 
$$
\|(I+E)^{-\alpha}\|_{S_p}^p
= \sum\limits_{j=0}^{\infty}d_j (1+\lambda_j)^{-\alpha p},
$$
which is finite if and only if $\alpha p>\frac{n}{\nu}$ by (\ref{asymp2}),
implying the statement.
\end{proof}

If $P$ is a pseudo-differential operator on $M$, 
for a function (or distribution) on $M\times M$, 
we will use the notation
$P_y K(x,y)$ to indicate that the operator $P$ is acting on the $y$-variable, the second factor 
of the product $M\times M$. 
For a positive elliptic operator $P\in \Psi^\nu_{+e}(M)$, 
by the elliptic regularity, the Sobolev space $H^\mu(M)$ can be characterised
as the space of all distributions $f\in\mathcal D'(M)$ such that
$(I+P)^{\frac\mu\nu} f\in L^2(M)$, and this characterisation is independent of the choice of operator $P$
(see also Lemma \ref{invellipt}).

\medskip
We now define Sobolev spaces $H^{\mu_1,\mu_2}_{x,y}(M\times M)$
of mixed regularity $\mu_1,\mu_2\geq 0$. We observe that for $K\in L^2(M\times M)$,
we have
\begin{align*} \|K\|_{L^2(M\times M)}^2=\int\limits_{M\times M}|K(x,y)|^2 dxdy
=\int\limits_{M}\left(\int\limits_{M}|K(x,y)|^2dy\right)dx,
\end{align*}
or we can also write this as
\beq\label{fubini} K\in L^2(M\times M) 
\Longleftrightarrow
K\in L_x^2(M,L_y^2(M)).\eq
In particular, this means that $K_x$ defined by $K_x(y)=K(x,y)$ is well-defined for almost every 
$x\in M$ as a function in $L^2_y(M)$. 

\begin{defn}\label{def1} 
Let $K\in L^2(M\times M)$ and let $\mu_1,\mu_2\geq 0$.
We say that $K\in H^{\mu_1,\mu_2}_{x,y}(M\times M)$ if 
$K_x\in H^{\mu_2}(M)$ for almost all $x\in M$, and if 
$(I+P_x)^{\frac{\mu_1}{\nu}} K_x\in L^2_x(M,H^{\mu_2}_y(M))$ for some 
$P\in \Psi^\nu_{+e}(M)$, $\nu>0$.
We set
$$\|K\|_{H^{\mu_1,\mu_2}_{x,y}(M\times M)}:=
\left( \int_M\|(I+P_x)^{\frac{\mu_1}{\nu}} K_x\|_{H^{\mu_2}(M)}^2 dx\right)^{1/2}.
$$
\end{defn}
By the elliptic regularity it follows that different choices of 
operators $P\in \Psi^\nu_{+e}(M)$, $\nu>0$, give equivalent norms on 
the space $H^{\mu_1,\mu_2}_{x,y}(M\times M)$. Thus, for
operators $E_j\in\Psi^{\nu_j}_{+e}(M)$
($j=1,2$) with $\nu_j>0$, we can formulate Definition \ref{def1} 
in an alternative (and perhaps more practical) way:

\begin{defn}\label{def2} 
For
operators $E_j\in\Psi^{\nu_j}_{+e}(M)$
($j=1,2$) with $\nu_j>0$, we define
\begin{equation}\label{EQ:mixed-Sobolev}
K\in H^{\mu_1,\mu_2}_{x,y}(M\times M) \Longleftrightarrow
(I+E_1)_x^{\frac{\mu_1}{\nu_1}} (I+E_2)_y^{\frac{\mu_2}{\nu_2}} K \in L^2(M\times M),
\end{equation}
where the expression on the right hand side means that we are applying
pseudo-differential operators on $M$ separately in $x$ and $y$. We note that
these operators commute since they are acting on different sets of variables
of $K$. 
\end{defn}
As we have noted above, the definition does not depend on a particular choice of
operators $E_j\in\Psi^{\nu_j}_{+e}(M)$, with the norms of $K$ induced by
\eqref{EQ:mixed-Sobolev} being all equivalent to each other and to that
in Definition \ref{def1}.
In Proposition \ref{PROP:mixed-Sobolev} we establish some properties
of the spaces $ H^{\mu_1,\mu_2}_{x,y}$, namely, we will show the
inclusions between the mixed and the standard Sobolev spaces on
the compact (closed) manifold $M\times M$ as 
$$
H^{\mu_{1}+\mu_{2}}(M\times M)\subset H^{\mu_{1},\mu_{2}}_{x,y}(M\times M)
\subset H^{\min (\mu_{1},\mu_{2})}(M\times M),
$$
for any $\mu_{1},\mu_{2}\geq 0$.

\medskip
We will now give our main criteria for Schatten classes. 
\begin{thm} \label{ext322} 
Let $M$ be a closed manifold of dimension $n$ and let $\mu_1, \mu_2 \geq 0$. 
Let 
$K\in L^2(M\times M)$ be such that $K\in H^{\mu_1,\mu_2}_{x,y}(M\times M)$. Then
the integral operator $T$ on 
 $L^2(M)$, defined by
 \[(Tf)(x)=\int\limits_{M} K(x,y)f(y)dy,\]
is in the Schatten classes $S_r(L^2(M))$ for $r>\frac{2n}{n+2(\mu_1+\mu_2)}$.
\end{thm}

\begin{rem} 
The value for $r$ comes from the relation 
\[\frac{1}{r}=\half+\frac {1}{p_1}+\frac {1}{p_2} ,\]
for some $0<p_1,p_2<\infty$, where the condition $r>\frac{2n}{n+2(\mu_1+\mu_2)}$ comes from $\mu_j>\frac{n}{p_j\nu_j}$ by a suitable 
 application of \eqref{EQ:ell-powers}. 
Also, since then $r=\frac{2p_1p_2}{p_1p_2+2(p_1+p_2)}$, the range for $r$ is the interval $(0,2)$ 
since, in general, $0<p_j<\infty .$ Therefore, Theorem \ref{ext322} provides
a sufficient condition for Schatten classes $S_r$ for $0<r<2$. For $\mu_1,\mu_2=0$
the conclusion is trivial and can be sharpened
to include $r=2$.
\end{rem}

\begin{rem}
We note that for $\mu_{1}=0$, Theorem \ref{ext322}
says that for $K\in L^{2}(M, H^{\mu}(M))$, we have that 
the corresponding operator $T$ satisfies
$T\in S_{r}$ for $r>\frac{2n}{n+2\mu}$. In this case no regularity
in the $x$-variable is imposed on the kernel. 

We also note that the `dual' result with $\mu_{2}=0$ imposing no regularity of $K$
with resect to $y$ also follows directly from it by considering the adjoint operator $T^{*}$
and using the equality $\|T^{*}\|_{S_{r}}=\|T\|_{S_{r}}.$
\end{rem}

\begin{proof}[Proof of Theorem \ref{ext322}] Let, for example, 
$E=(I+\Delta_{M})^{\half}$, where $\Delta_{M}$   is a 
positive definite elliptic differential operator of order $2$, and $E=\overline{E}$.
The existence of such $\Delta_{M}$ follows, for example, from the Whitney embedding theorem.

$(i)$  We first suppose that $\mu_1,\mu_2>0$. 
By Proposition \ref{PROP:elliptic} with $\alpha=\mu_2$ and $\nu=1$ we get, in particular, that 
$E_y ^{-\mu_2}\in S_{p_2}(L^2(M))$ for $\mu_2>\frac{n}{p_2}$, the first ingredient in the proof. 



Now let us consider the kernel $B(x,y)$ of the operator  $E_y ^{-\mu_2}\in S_{p_2}(L^2(M))$ for $\mu_2>\frac{n}{p_2}$. If $f\in L^2(M)$, let  
\[g(y):=\int\limits_{M}B(y,z)f(z)dz,\]
so that
\[E_y^{\mu_2}g(y)=f(y).\]

First we note that by \eqref{EQ:mixed-Sobolev} the condition $K\in H^{\mu_1,\mu_2}_{x,y}(M\times M)$ 
can be written as $E_x ^{\mu_1}E_y ^{\mu_2}K\in L^2(M\times M)$, with $\nu_1= \nu_2=1$. Since $E_x ^{\mu_1}E_y ^{\mu_2}K\in L^2(M\times M)$, we have
$E_x ^{\mu_1}E_y ^{\mu_2}K(x,\cdot)\in L_y^2(M)$ for almost every $x$, and this
 fact will justify the use of scalar products in the next argument.  

We observe that 
\begin{align*} E_x ^{\mu_1}Tf(x)=&\int_{M}E_x ^{\mu_1}K(x,y)f(y)dy\\
=&\int_{M}E_x ^{\mu_1}K(x,y)E_y ^{\mu_2}g(y)dy\\
=&(E_x ^{\mu_1}K(x,\cdot),E_y ^{\mu_2}\overline{g})_{L^2(M)}\\
=&(E_y ^{\mu_2}E_x ^{\mu_1}K(x,\cdot),\overline{g})_{L^2(M)}\\
=&(E_x ^{\mu_1}E_y ^{\mu_2}K(x,\cdot),\overline{g})_{L^2(M)}\\
=&\int_{M}E_x ^{\mu_1}E_y ^{\mu_2}K(x,y)g(y)dy\\
=&\int_{M}E_x ^{\mu_1}E_y ^{\mu_2}K(x,y)\left(\int_{M}B(y,z)f(z)dz\right)dy.
\end{align*} 

Now, setting $A(x,y):=E_x ^{\mu_1}E_y ^{\mu_2}K(x,y)$, we have shown that
\[
E_x ^{\mu_1}Tf(x)=\int_{M}E_x ^{\mu_1}K(x,y)f(y)dy=\int_{M}A(x,y)\left(\int_{M}B(y,z)f(z)dz\right)dy,
\]
thus we have factorised the operator $E_x ^{\mu_1}T$ with $A\in S_2$ and $B\in S_{p_2}$. 
By (\ref{inn}) we get that $E_x ^{\mu_1}T\in S_t$ with $\frac{1}{t}=\frac{1}{2}+\frac{1}{p_2} $.
 
On the other hand, since $E_x ^{-\mu_1}\in S_{p_1}$ for $\mu_1>\frac{n}{p_1} $, we obtain
\[T=E_x ^{-\mu_1}E_x ^{\mu_1}T\in S_r,\]
with
\[\frac 1r=\frac 1{p_1}+\frac 1t=\half+\frac 1{p_1}+\frac 1{p_2}.\]
Using inequalities $p_1>\frac{n}{\mu_1}$ and $p_2>\frac{n}{\mu_2}$, this is equivalent to 
\[r>\frac{2n}{n+2(\mu_1+\mu_2)}.\]
$(ii)$ If $\mu_1=0$ and $\mu_2>0$, just by removing the operator $E_x ^{\mu_1}$  in the argument above we get the desired result.\\
$(iii)$ If $\mu_1>0$ and $\mu_2=0$. This is a consequence of case $(ii)$ proceeding by duality, considering the adjoint of the operator $T$ and applying the fact that $\|T\|_{S_r}=\|T^*\|_{S_r}$. 
\end{proof}

\section{Trace class operators and their traces}
\label{SEC:trace-class}

We shall now briefly recall the averaging process which is required for the study of trace formulae for kernels with discontinuities along the diagonal. We start by defining the martingale
maximal function. Let $({\Omega},{\mathcal{M}},\mu)$ be a $\sigma$-finite measure space and
let $\{\mathcal{M}_j\}_{j}$ be a sequence of sub-$\sigma$-algebras such that
$$
\mathcal{M}_j\subset\mathcal{M}_{j+1}\,\,\textrm{ and }{\mathcal{M}}=\bigcup\limits_{j}\mathcal{M}_j.
$$
In order to define conditional expectations we assume that $\mu$ is $\sigma$-finite on each $\mathcal{M}_j$. In that case, 
if $f\in L^p(\mu)$, then $E(f|\mathcal{M}_n)$ exists. We say that a sequence $\{f_j\}_{j}$ of functions on $\Omega$ is a {\em martingale} if each $f_j$ is $\mathcal{M}_j$-measurable and
\beq E(f_j|\mathcal{M}_k)=f_k\,\mbox{ for }k<j.\eq
 In order to obtain a generalisation of the Hardy-Littlewood maximal function we consider the particular case of martingales generated by a single $\mathcal{M}$-measurable function $f$. The {\em martingale maximal function} is defined by
 \beq 
 Mf(x):=\sup\limits_{j}E(|f|\left .\right|\mathcal{M}_j)(x).
 \eq
This martingale can be defined, in particular, on a second countable topological space endowed
with a $\sigma$-finite Borel measure. For our purposes in the study of the kernel the sequence 
of $\sigma$-algebras is constructed from a corresponding increasing sequence of 
partitions $\mathcal{P}_j\times\mathcal{P}_j$ of $\Omega\times\Omega$ with 
$\Omega=M$, the closed manifold. 

Now, for each $(x,y)\in M\times M$ there is a unique $C_j(x)\times C_j(y)\in  \mathcal{P}_j\times\mathcal{P}_j$ containing $(x,y)$. Those sets $C_j(x)$ replace the cubes in $\ern$ in the definition of the classical Hardy-Littlewood maximal function. We refer to Doob 
\cite{Doob:bk-measure-theory-1994} for more details on the martingale maximal function
and its properties.

We denote by $A_j^{(2)}$ the averaging operators on $\Omega\times\Omega$:  Let $K\in
L_{loc}^1(\mu\otimes\mu)$, then the averaging $A_j^{(2)}$ is defined $\mu\otimes\mu$-almost everywhere (cf. \cite{bri:k2}) by 
\beq 
A_j^{(2)}K(x,y):=\frac{1}{\mu(C_j(x))\mu(C_j(y))}\int\limits_{C_j(x)}\int\limits_{C_j(y)}K(s,t)d\mu(t)d\mu(s).
\eq

The averaging process will be applied to the kernels $K(x,y)$ of our operators.  
As a consequence of the fundamental properties of the martingale maximal function it can be deduced that 
$$
\widetilde{K}(x,y):=\lim_{j\rightarrow\infty} A_j^{(2)}K(x,y),
$$ 
is defined almost everywhere and that it agrees with $K(x,y)$ in the points of continuity.  
 We observe that if $K(x,y)$ is the integral kernel of a trace class operators, 
 then $K(x,y)$ is, in particular, square integrable, and hence by the 
 H{\"o}lder inequality it is integrable on the compact manifold $M\times M.$


\medskip
In the sequel in this section, we can always assume that $K\in L^2(M\times M)$ since
it is not restrictive because the trace class is included in the Hilbert-Schmidt class, and
the square integrability of the kernel is then a necessary condition. 

\medskip
As a corollary of Theorem \ref{ext322}, for the trace class operators we have:
\begin{cor} \label{ext322aw} 
Let $M$ be a closed manifold of dimension $n$ and let $K\in L^2(M\times M)$, 
$\mu_1, \mu_2 \geq 0$, be such that  $\mu_1+\mu_2>\frac n2$ and 
$K\in H^{\mu_1,\mu_2}_{x,y}(M\times M)$. Then
the integral operator $T$ on 
 $L^2(M)$, defined by
 \[(Tf)(x)=\int\limits_{M} K(x,y)f(y)dy\]
is trace class and its trace is given by
\begin{equation}\label{EQ:trace}
\Tr(T)=\int\limits_M\widetilde{K}(x,x)dx.
\end{equation}
\end{cor}
\begin{proof} We observe that to get $r=1$ from Theorem \ref{ext322}, we require the following inequality to hold:
\[1>\frac{2n}{n+2(\mu_1+\mu_2)}.\]
But this is equivalent to $\mu_1+\mu_2>\frac n2$. The trace formula is a consequence of (\ref{f1}).
\end{proof}

Corollary \ref{ext322aw} improves, in particular, Proposition \ref{ext1}: 
\begin{cor}\label{mix3} 
Let $M$ be a smooth closed manifold of dimension $n$. 
Let $K\in L^2(M\times M)$ be such that $K\in H^{\mu}(M\times M)$ for 
 $\mu>\frac{n}{2}$. Then the integral operator $T$ on 
 $L^2(M)$, defined by
 \[(Tf)(x)=\int\limits_{M}K(x,y)f(y)dy,\]
is trace class on $L^{2}(M)$ and its trace is given by \eqref{EQ:trace}.
\end{cor}
Indeed, Corollary \ref{mix3} follows from Corollary \ref{ext322aw} and the 
inclusion $H^{\mu}(M\times M)\subset H^{0,\mu}_{x,y}(M\times M)$, the latter being a
special case of the following inclusions between usual and mixed Sobolev spaces:

\begin{prop}\label{PROP:mixed-Sobolev}
Let $M$ be a smooth closed manifold . Then
we have the inclusions
\begin{equation}\label{EQ:Sobolevs}
H^{\mu_{1}+\mu_{2}}(M\times M)\subset H^{\mu_{1},\mu_{2}}_{x,y}(M\times M)
\subset H^{\min (\mu_{1},\mu_{2})}(M\times M),
\end{equation}
for any $\mu_{1},\mu_{2}\geq 0$.
\end{prop}
\begin{proof}
Let $\Delta$ be a second order positive elliptic differential operator on $M$
(such an operator exists e.g. by the Whitney embedding theorem),
and let $e_{j}^{k}$ denote the orthonormal basis and $\lambda_{j}$ the corresponding
eigenvalues, leading to the discrete Fourier analysis associated to $\Delta$ as described
in Section \ref{SEC:Fourier}. Then the products
$e_{j}^{k}(x)e_{l}^{m}(y)$ give rise to an orthonormal basis in $L^{2}(M\times M)$.
Consequently, using \eqref{EQ:mixed-Sobolev}
with $E_{1}=E_{2}=\Delta$, we see that 
$K\in H^{\mu_{1},\mu_{2}}_{x,y}(M\times M)$ is equivalent to the condition
\begin{equation}\label{EQ:K1}
\sum_{j=0}^{\infty}\sum_{k=1}^{d_{j}}\sum_{l=0}^{\infty}\sum_{m=1}^{d_{l}}
(1+{\lambda_{j}})^{\mu_{1}} (1+{\lambda_{l}})^{\mu_{2}}
|\widehat{K}(j,k,l,m)|^{2}<\infty,
\end{equation}
where the Fourier coefficients $\widehat{K}(j,k,l,m)$ are determined by the
Fourier series
$$
K(x,y)=\sum_{j=0}^{\infty}\sum_{k=1}^{d_{j}}\sum_{l=0}^{\infty}\sum_{m=1}^{d_{l}}
\widehat{K}(j,k,l,m)e_{j}^{k}(x)e_{l}^{m}(y).
$$
On the other hand, we note that 
$(1+\Delta_{x}+\Delta_{y})e_{j}^{k}(x)e_{l}^{m}(y)=(1+\lambda_{j}+\lambda_{l})
e_{j}^{k}(x)e_{l}^{m}(y)$, which implies that 
$K\in H^{\mu}(M\times M)$ is equivalent to
\begin{equation}\label{EQ:K2}
\sum_{j=0}^{\infty}\sum_{k=1}^{d_{j}}\sum_{l=0}^{\infty}\sum_{m=1}^{d_{l}}
(1+{\lambda_{j}}+{\lambda_{l}})^{\mu}
|\widehat{K}(j,k,l,m)|^{2}<\infty.
\end{equation}
Comparing the expressions in \eqref{EQ:K1} and \eqref{EQ:K1}, and using the
inequalities
$$
(1+{\lambda_{j}}+{\lambda_{l}})^{\min (\mu_{1},\mu_{2})}\leq
C(1+{\lambda_{j}})^{\mu_{1}} (1+{\lambda_{l}})^{\mu_{2}}
\leq 
C(1+{\lambda_{j}}+{\lambda_{l}})^{\mu_{1}+\mu_{2}},$$ 
we obtain the
inclusions \eqref{EQ:Sobolevs}.
\end{proof}

We also obtain some corollaries in terms of the derivatives of the kernel. We denote 
by $C_x^\alpha C_{y}^{\beta}(M\times M)$ the space of functions of class $C^{\beta}$ 
with respect to $y$ and $C^\alpha$ with respect to $x$.

 \begin{cor} \label{cor25} 
 Let $M$ be a closed manifold of dimension $n$. Let   
 $K\in C_x^{\ell_1} C_{y}^{\ell_2}(M\times M)$ some even integers $\ell_1,\ell_2\in 2\mathbb N_0$
 such that $\ell_1+\ell_2>\frac n2$. Then
 the integral operator $T$ on 
 $L^2(M)$, defined by
 \[(Tf)(x)=\int_{M}K(x,y)f(y)dy,\]
is in $S_1(L^2(M))$, and its trace is given by 
\begin{equation}\label{EQ:trace2}
\Tr(T)=\int\limits_M K(x,x)dx.
\end{equation}
\end{cor}

\begin{proof} 
Let $\lapm$ be an elliptic positive definite second order differential operator on $M$, then 
$(I+\lapm)_x ^{\frac{\ell_1}{2}}(I+\lapm)_y ^{\frac{\ell_2}{2}}K\in C(M\times M)\subset L^2(M\times M)$. 
Now, by observing that $\ell_1+\ell_2>\frac{n}{2}$ the result follows from Corollary \ref{ext322aw} . 
\end{proof}

\begin{rem} The index $\frac n2$ in Corollary \ref{cor25} is sharp. Indeed, for the torus $\tn$ with $n$ even, there exist a function $\chi$ of class $C^{\frac n2}$ such that the series of Fourier coefficients diverges (cf. \cite{ste-we:fa}, Ch. VII; \cite{wai:trig}). By considering the convolution kernel $K(x,y)=\chi(x-y)$, the singular values of the operator $T$ given by $Tf=f*\chi$ agree with the absolute values of the Fourier coefficients. Hence, $T\notin S_1(L^2(\tn))$ but $K\in C^{\frac n2}(M\times M)$
(we can think of e.g. $\ell_1=0$ and $\ell_2=\frac n2$ in Corollary  \ref{cor25}).
On the other hand, concerning necessary conditions on the kernel,
writing the convolution operator $T$ in the pseudo-differential form
$$
Tf(x)=\sum_{\xi\in\mathbb Z^{n}} e^{i x\cdot\xi} \sigma(\xi)\widehat{f}(\xi)
$$
with $\sigma(\xi)=\widehat{\chi}(\xi)$,
it can be shown that $\sum\limits_{\xi\in\mathbb Z^{n}}|\sigma(\xi)|<\infty $ if and only if the corresponding pseudo-differential operator $T_{\sigma}$ is trace class on $L^2(\Tn)$ (cf. \cite{dr13:schatten}). Hence, when dealing with a multiplier we can deduce that if $T_{\sigma}$ is trace class then its kernel is continuous. This can be obtained from the formula for the convolution kernel
\[K(x,y)=\sum\limits_{\xi\in\mathbb Z^{n}}e^{i(x-y)\xi}\sigma(\xi)\]
and the summability of $\sigma$. Therefore, the continuity of kernels is a necessary condition for traceability of convolution operators on $\Tn$. However, as we note from the example (\ref{EQ:Carleman2}) on $M=\T$ the convolution kernel $K(x,y)=\varkappa(x-y)$ is continuous but the corresponding operator is not trace class. 
\end{rem}

We now make some remarks about the relation between the trace class property and the 
Fourier coefficients of the kernel, in the sense of Section \ref{SEC:Fourier}.

\medskip
The main idea in the proof by Sugiura of \eqref{EQ:Sugiura-trace}
(and then also of Proposition \ref{ext1})
consists in exploiting the underlying relation between 
the convergence of the series of Fourier coefficients of the kernel $K\in C^2(\mathbb{T}^1\times\mathbb{T}^1)$ and the traceability.
 This link is basically reduced to the application of the following  classical result, see e.g. \cite{sugiura:book}: 
 
\begin{lem} \label{lem11} 
Let $H$ be a separable Hilbert space.  If a bounded linear operator $T$ on $H$ 
 satisfies 
  \[\sum\limits_{m,j\in\ene }|\p{T\phi_j,\phi_m}|<\infty\]
 for a fixed orthonormal basis $(\phi_j)_j$, then $T$ is trace class.
\end{lem}

By choosing an orthonormal basis $\{\phi_j\}$ consisting of eigenfunctions of the 
Laplacian on $\mathbb{T}^1\times\mathbb{T}^1$ one can prove that 
the Fourier coefficients of $K$ agree with the values $\p{T\phi_j,\phi_m}$. 
On the other hand, it can be shown that the series of Fourier 
 coefficients converges absolutely and hence the traceability follows. The proof can be extended to smooth closed manifolds by using elliptic positive pseudo-differential operators instead of Laplacians on $M\times M$, and associating a discrete Fourier analysis to the cross product $M\times M$. However, this method leads to a weaker result, by furnishing the condition $K=K(x,y)\in C^{\nu}(M\times M)$ with $\nu >n$. We have improved that kind of result by obtaining a sharp condition on the regularity of $K$ with respect to $y$ as given in Corollary \ref{cor25}
 (with $\ell_{1}=0$ and $\ell_{2}>n/2$). 
 
However, for the sake of completeness we establish below a result concerning the convergence of series of Fourier coefficients. The related problem concerning the convergence of Fourier series on compact manifolds has been studied by
Taylor \cite{tay:fc}. Taylor's paper also included a special version for compact Lie groups. Similar results for compact connected Lie groups from a different approach were obtained by Sugiura in \cite{sugiura:p1}. In order to study such kind of convergence for a kernel $K$ on $M\times M$ we will first apply the convergence criterion (\ref{asymp2}) to the manifold $M\times M$. In the following lemma we will associate to an elliptic positive  pseudo-differential operator ${E}$ of order $\nu$ on $M\times M$ an orthonormal basis of $L^2(M\times M)$ which we denote by $\{e_{\ell}^m\}_{\ell\geq 0}^{1\leq m\leq d_{\ell}}$. 

\begin{lem} \label{lem22a} 
Let $M$ be a compact smooth manifold of dimension $n$, and let ${E}\in\Psi_{+e}^{\nu}(M\times M)$ with $\nu>n$. Let $K\in L^2(M\times M)$ be such that $EK\in L^2(M\times M)$.  
Let $\lambda_{\ell}$ be the eigenvalues of ${E}$ and let $d_{\ell}$ be the corresponding multiplicities.
With respect to the orthonormal basis $\{e_{\ell}^m\}_{\ell\geq 0}^{1\leq m\leq d_{\ell}}$ of $L^2(M\times M)$, $K(x,y)$ can be written as
\begin{equation}\label{EQ:Kernel-FT}
 K(x,y)=\sum\limits_{\ell=0}^{\infty}\sum\limits_{m=1}^{d_{\ell}}\widehat{K}(\ell,m)e_{\ell}^m(x,y),
\end{equation}
with 
$\sum\limits_{\ell=0}^{\infty}\sum\limits_{m=1}^{d_{\ell}} |\widehat{K}(\ell,m)|<\infty$,
and the condition $EK\in L^2(M\times M)$ is independent of the choice of the 
operator $E$ in $\Psi_{+e}^{\nu}(M\times M)$, and can be expressed as the Sobolev space
condition $K\in H^{\nu}(M\times M)$. 
\end{lem}

\begin{proof}
We define $h(x,y):=EK(x,y)$, which is in $L^2(M\times M)$. 
With respect to the orthonormal basis $\{e_{\ell}^m\}_{\ell\geq 0}^{1\leq m\leq d_{\ell}}$ of $L^2(M\times M)$ the kernel $K(x,y)$ can be written as
\eqref{EQ:Kernel-FT}.
At the same time, the Fourier coefficients of $h$ satisfy, by the Plancherel formula
\eqref{EQ:Plancherel},
\beq \sum\limits_{\ell=0}^{\infty}\sum\limits_{m=1}^{d_{\ell}}|\widehat{h}(\ell,m)|^2= \|h\|_{L^2(M\times M)}^2,\label{ineq66}\eq
and
\begin{align*}\widehat{h}(\ell,m)={{\widehat{EK}}}(\ell,m)
=\lambda_{\ell}\widehat{K}(\ell,m),
\end{align*}
for $0\leq \ell <\infty$ and $1\leq m\leq d_{\ell}.$

For the Fourier coefficients of $K$, by the Cauchy-Schwarz inequality, (\ref{ineq66}) and the inequality (\ref{asymp2}) applied to ${E}$ on $M\times M $ we obtain  

\begin{multline*}
\sum\limits_{\ell=0}^{\infty}\sum\limits_{m=1}^{d_{\ell}}|\widehat{K}(\ell,m)|
\leq C\sum\limits_{\ell=0}^{\infty}\sum\limits_{m=1}^{d_{\ell}}\lambda_{\ell}\frac{|\widehat{K}(\ell,m)|}{1+\lambda_{\ell}}
=C\sum\limits_{\ell=0}^{\infty}\sum\limits_{m=1}^{d_{\ell}}\frac{|\widehat{h}(\ell,m)|}{1+\lambda_{\ell}}\\
\leq  C\left(\sum\limits_{\ell=0}^{\infty}\sum\limits_{m=1}^{d_{\ell}}|\widehat{h}(\ell,m)|^2\right)^\half\left(\sum\limits_{\ell}d_{\ell}(1+\lambda_{\ell})^{-2}\right)^\half
\leq  C\|h\|_{L^2},
\end{multline*}
where for the application of (\ref{asymp2}), we note that the convergence of the series $\sum\limits_{\ell}d_{\ell}(1+\lambda_{\ell})^{-2}$ on $M\times M$ with $q=2$ 
 is equivalent to the condition $\nu>n$. The independence of the choice of an elliptic positive  pseudo-differential operator of order $\nu$ is an immediate consequence of Lemma \ref{invellipt} for the manifold $M\times M$.
\end{proof}

\section{Schatten classes on compact Lie groups} 
\label{SEC:Lie-groups}

In this section we consider the conditions for Schatten classes for operators on compact Lie groups.
We show that the conditions on the kernel can be also formulated in terms of hypoelliptic
operators. This is done by combining the factorisation method used in the previous 
sections with recent results \cite{dr13:schatten} by the authors on characterisation of invariant operators in Schatten classes on compact Lie groups. 
We start by describing the basic concepts we will require for this setting. 

Let $G$ be a  compact Lie group of dimension $n$ with 
the normalised Haar measure $dx$. 
Let $\widehat{G}$ denote the set of equivalence classes of continuous irreducible unitary 
representations of $G$. Since $G$ is compact, the set $\widehat{G}$ is discrete.  
For $[\xi]\in \widehat{G}$, by choosing a basis in the representation space of $\xi$, we can view 
$\xi$ as a matrix-valued function $\xi:G\rightarrow \ce^{d_{\xi}\times d_{\xi}}$, where 
$d_{\xi}$ is the dimension of the representation space of $\xi$. 
By the Peter-Weyl theorem the collection
$$
\left\{ \sqrt{d_\xi}\,\xi_{ij}: \; 1\leq i,j\leq d_\xi,\; [\xi]\in\Gh \right\}
$$
is the orthonormal basis of $L^2(G)$.
If $f\in L^1(G)$ we define its global Fourier transform at $\xi$ by 
 \[\widehat{f}(\xi):=\int_{G}f(x)\xi(x)^*dx.\]  
If $\xi$ is a matrix representation, we have $\widehat{f}(\xi)\in\ce^{d_{\xi}\times d_{\xi}} $. The Fourier inversion formula is a consequence of the Peter-Weyl theorem, and we have
\beq 
f(x)=\sum\limits_{[\xi]\in \widehat{G}}d_{\xi} \Tr(\xi(x)\widehat{f}(\xi)).
\eq

For each $[\xi]\in \widehat{G}$, the matrix elements of $\xi$ are the eigenfunctions for the Laplacian
 $\mathcal{L}_G$ 
(or the Casimir element of the universal enveloping algebra), with the same eigenvalues which we denote by 
$-\lambda^2_{[\xi]}$, so that we have
\begin{equation}\label{EQ:Lap-lambda}
-\mathcal{L}_G\xi_{ij}(x)=\lambda^2_{[\xi]}\xi_{ij}(x)\qquad\textrm{ for all } 1\leq i,j\leq d_{\xi}.
\end{equation} 
The weight for measuring the decay or growth of Fourier coefficients in this setting is 
$\jp{\xi}:=(1+\lambda^2_{[\xi]})^{\half}$, the eigenvalues of the 
(positive) elliptic first-order pseudo-differential operator 
$(I-\mathcal{L}_G)^{\half}$.
The Parseval identity takes the form 
$$
\|f\|_{L^2(G)}= \left(\sum\limits_{[\xi]\in \widehat{G}}d_{\xi}\|\widehat{f}(\xi)\|^2_{\HS}\right)^{\half},\quad
\textrm{ where }
\|\widehat{f}(\xi)\|^2_{\HS}=\Tr(\widehat{f}(\xi)\widehat{f}(\xi)^*),$$
which gives the norm on 
$\ell^2(\widehat{G})$. 

For a linear continuous operator $A$ from $C^{\infty}(G)$ to $\mathcal{D}'(G) $ 
we define  its {\em matrix-valued symbol} $\sigma_A(x,\xi)\in\cdxi$ by 
\begin{equation}\label{EQ:A-symbol}
\sigma_A(x,\xi):=\xi(x)^*(A\xi)(x)\in\cdxi.
\end{equation}
Then one has (\cite{rt:book}, \cite{rt:groups}) the global quantization
\begin{equation}\label{EQ:A-quant}
Af(x)=\sum\limits_{[\xi]\in \widehat{G}}d_{\xi}\Tr(\xi(x)\sigma_A(x,\xi)\widehat{f}(\xi))
\end{equation}
in the sense of distributions, and the sum is independent of the choice of a representation $\xi$ from each 
equivalence class 
$[\xi]\in \widehat{G}$. If $A$ is a linear continuous operator from $C^{\infty}(G)$ to $C^{\infty}(G)$,
the series \eqref{EQ:A-quant} is absolutely convergent and can be interpreted in the pointwise
sense. We will also write $A=\Op(\sigma_A)$ for the operator $A$ given by
the formula \eqref{EQ:A-quant}.
We refer to \cite{rt:book, rt:groups} for the consistent development of this quantization
and the corresponding symbolic calculus.


In the subsequent Part II of this paper, we will relate the Fourier 
and symbolic analysis on general closed
manifolds to those on compact Lie groups. So, now we will only
concentrate on showing that the conditions
for Schatten classes can be also formulated by examining the regularity of the kernel under the
action of non-elliptic but hypoelliptic operators.

Instead of Proposition \ref{PROP:elliptic} for elliptic operators as the starting point, here we will use its
analogue for hypoelliptic operators established in 
\cite{dr13:schatten}, for example its analogue for sub-Laplacians on
compact Lie groups $G$.
Let us write
\[\lap=X_1^2+\cdots +X_{n-1}^2+X_{n}^2,\]  
\begin{equation}\label{EQ:def-subL}
\subll=X_1^2+\cdots +X_{n-1}^2,
\end{equation}
for a basis $X_1,\ldots,X_n$ of left-invariant
vector fields on the Lie algebra $\mathfrak g$ of $G$, 
assuming that the span of the first commutators of $X_1,\ldots,X_{n-1}$ contains $X_n$.
Then it was shown in \cite{dr13:schatten} that 
\begin{equation}\label{EQ:subL-Schatten}
0<r<\infty \textrm{ and } \alpha r>2n 
\Longrightarrow (I-\mathcal L_{sub})^{-\alpha/2}\in S_r(L^{2}(G)).
\end{equation}
Here we can note that the powers of the hypoelliptic positive pseudo-differential operator 
$I-\mathcal L_{sub}$ are well-defined. There is a general theory, 
see e.g. \cite{Kumano-go-Tsutsumi:complex-powers-hypo-OJM-1973},
\cite{Hayakawa-Kumano-go:powers-PJA-1971},
or more recent results and references in \cite{Buzano-Nicola:powers-hypo-JFA-2007}.
However, if we observe that the matrix symbol of the sub-Laplacian
(as well as the symbol of the operator $\mathcal H_{\gamma}$ is the sequel) are diagonal,
a complex power of such operator may be defined by the quantization formula
\eqref{EQ:A-quant} using the matrix symbol being the corresponding
complex power of the diagonal symbol of the operator. 
For left-invariant operators with diagonal matrix symbols
(such as $\mathcal L_{sub}$ or $\mathcal H_{\gamma}$)
all such approaches yield the same 
operators (see e.g. \cite{Ruzhansky-Wirth:functional-calculus} for more details).

Consequently, arguing in the same way as in the proof of 
Theorem \ref{ext322} we obtain:

\begin{cor} \label{COR:sub-Lap-G} 
Let $G$ be a compact Lie group of dimension $n$, and let $\subll$ be a
sub-Laplacian as in \eqref{EQ:def-subL}.
Let $K\in L^2(M\times M)$ be such that 
$(I-\subll)_x^{\mu_1/2}(I-\subll)_y^{\mu_2/2}K\in L^2(G\times G)$ for some $\mu_1, \mu_2\geq 0$. 
Then the integral operator $T$ on 
 $L^2(G)$, defined by
 \[(Tf)(x)=\int_{G} K(x,y)f(y)dy,\]
is in the Schatten classes $S_r(L^2(G))$ for $r>\frac{2n}{n+(\mu_1+\mu_2)}$.
\end{cor}
\begin{proof}
We argue similar to the proof of Theorem \ref{ext322}.
In particular, we know from \eqref{EQ:subL-Schatten} that 
$(I-\mathcal L_{sub})^{-\mu_j/2}\in S_{p_j}$ for $\mu_j p_j>2n\,\,(j=1,2)$.
From the relation  $\frac{1}{r}=\frac{1}{2}+\frac{1}{p_1}+\frac{1}{p_2}$ and $p_j>\frac{2n}{\mu_j}$, we
get that under the condition 
$r>\frac{2n}{n+(\mu_1+\mu_2)}$ the operator $T$ belongs to the Schatten class $S_{r}$ on
$L^{2}(G)$.
\end{proof}

As it was noted in \cite{dr13:schatten}, the implication
\eqref{EQ:subL-Schatten} can be improved for particular groups
using their particular structure. For example, for the compact Lie group
$\SU2$ we have, for three left-invariant vector fields $X,Y,Z$ that
$[X,Y]=Z$, and so with $\mathcal L_{sub}=X^{2}+Y^{2}$ we have
\begin{equation}\label{EQ:subL-Schatten-SU2}
0<r<\infty \textrm{ and } \alpha r>4 
\Longrightarrow (I-\mathcal L_{sub})^{-\alpha/2}\in S_r(L^{2}(\SU2)).
\end{equation}
The same is true for $\mathbb S^{3}\simeq \SU2$ considered as the compact
Lie group with the quaternionic product.
Using this instead of \eqref{EQ:subL-Schatten}, we get a refinement of
Corollary \ref{COR:sub-Lap-G} in the setting of $\mathbb S^{3}\simeq \SU2$:

\begin{cor} \label{COR:sub-Lap-SU2} 
Let $K\in L^2(\mathbb S^{3}\times \mathbb S^{3})$ be such that 
we have $(1-\subll)_x^{\mu_1/2}(1-\subll)_y^{\mu_2/2}K\in L^2(\mathbb S^{3}\times \mathbb S^{3})$ for some $\mu_1, \mu_2\geq 0$. 
Then the integral operator $T$ on 
 $L^2(\mathbb S^{3})$, defined by
 \[(Tf)(x)=\int_{\mathbb S^{3}} K(x,y)f(y)dy,\]
is in the Schatten classes $S_r(L^2(\mathbb S^{3}))$ for $r>\frac{4}{2+\mu_1+\mu_2}$.
The same result holds on compact Lie groups $\SU2$ and $\SO3$.
\end{cor}
\begin{proof}
Using \eqref{EQ:subL-Schatten-SU2} with $\alpha=\mu_j$, the assumption
$\mu_j p_j>4$ implies that 
$(I-\mathcal L_{sub})^{-\mu_j/2}\in S_{p_j}$ for $p_j>\frac{4}{\mu_j} (j=1,2)$.
From $\frac{1}{r}=\frac{1}{2}+\frac{1}{p_1}+\frac{1}{p_2}$ and $p_j>\frac{4}{\mu_j}$ we obtain the
condition $r>\frac{4}{2+\mu_1+\mu_2}$.
\end{proof}

We now show that instead of the sub-Laplacian other globally hypoelliptic operators
can be used, also those that are not necessarily covered by H\"ormander's sum of
the squares theorem. Instead of $\SU2$, for a change, we will formulate this for the group
$\SO3$ noting that, however, the same conclusion holds also on
$\SU2\simeq \mathbb S^{3}$.
To formulate and motivate the result, we first briefly introduce 
some more notation concerning the group $G=\SO3$ of the $3\times 3$ 
real orthogonal matrices 
of determinant one. For the details of the representation theory and the global 
quantization of operators on $\SO3$ we refer the reader to \cite[Chapter 12]{rt:book}. 
The unitary dual in this case of $G=\SO3$ can be identified as $\widehat{G}\simeq \ene_0=\ene\cup\{0\}$, so that
$$
\widehat{\SO3}=\{ [t^\ell]: t^\ell\in\C^{(2\ell+1)\times (2\ell+1)},  \ell\in \N_0\}.
$$
The dimension of each $t^\ell$ is $d_{t^\ell}=2\ell+1$.

As in the case of $\SU2$, let us fix three left-invariant vector fields $X, Y, Z$ on $\SO3$ 
associated to the derivatives with
respect to the Euler angles, 
so that we also have $[X,Y]=Z$,
see \cite{rt:book} or \cite{rt:groups} for the detailed expressions.

We will consider an {\em example} of an operator (on $\SO3$)
which is not covered by H\"ormander's sum of squares theorem.
Namely, we consider the following family of `Schr{\"o}dinger' differential operators
\[\mathcal H_{\gamma}=iZ-\gamma(X^2+Y^2),\]
for a parameter $0<\gamma<\infty.$ For $\gamma=1$ it 
was shown in \cite{Ruzhansky-Turunen-Wirth:arxiv} that
$\mathcal H_1+cI$ is globally hypoelliptic if and only if
$0\not\in\{c+\ell(\ell+1)-m(m+1): \ell\in\N, m\in\mathbb Z, |m|\leq\ell\}.$ 
%
%
It has been also shown in \cite[Section 4]{dr13:schatten} that, if $\gamma > 1$, then 
$I+\mathcal H_{\gamma}$ is globally hypoelliptic, and
$$(I+\mathcal H_{\gamma})^{-\alpha/2}\in S_p \textrm{ if and only if } \alpha p>4.$$
 
As a consequence of this and following the argument in the proof of Theorem \ref{ext322} 
with $I+\mathcal H_{\gamma}$ instead of $E=\lapm$ for the manifold $M=\SO3$, as well
as Corollary \ref{COR:sub-Lap-SU2}, we obtain:

\begin{cor} \label{ext322so3}  
Let $\gamma >1$. Let $K\in L^2(\SO3\times\SO3)$ be such that $(I+\mathcal H_{\gamma})_x ^{\mu_1/2}
(I+\mathcal H_{\gamma})_y ^{\mu_2/2}K\in L^2(\SO3\times\SO3)$ for some $\mu_1, \mu_2\geq 0$. 
Then the integral operator $T$ on 
 $L^2(\SO3)$ defined by
 \[(Tf)(x)=\int_{\SO3}K(x,y)f(y)dy,\]
is in $S_r$ on $L^{2}(\SO3)$ for $r>\frac{4}{2+\mu_1+\mu_2}$. 
\end{cor}
The same conclusion holds on $\SU2\simeq\mathbb S^3$. Again, if
$\mu_1=\mu_2=0$, the results have a trivial strengthening to include the case $r=2$.









%
%
%

\end{document}